\documentclass[a4paper,10pt]{amsart} 
\usepackage{amsmath,amssymb,amsfonts,graphicx,epsfig,color,url}
\usepackage{latexsym,mathrsfs,cite}

\usepackage{float}

\usepackage{hyperref} 
\usepackage{placeins}
\usepackage{comment}
\usepackage[title,titletoc,toc]{appendix}

\usepackage{graphicx}
\usepackage{subcaption}

\title[\resizebox{4.5in}{!}{Dispersive shocks and spectral analysis for linearized Quantum Hydrodynamics}]{Dispersive shocks and spectral analysis for linearized Quantum Hydrodynamics}
\author{Corrado Lattanzio, Pierangelo Marcati, Delyan Zhelyazov}
\address[Corrado Lattanzio]{DISIM, Department of Information Engineering, Computer Science and Mathematics \\ University of L'Aquila, Italy}
\email{corrado@univaq.it}
\address[Pierangelo Marcati]{GSSI, Gran Sasso Science Institute, L'Aquila, Italy}
\email{pierangelo.marcati@univaq.it}
\address[Delyan Zhelyazov]{GSSI, Gran Sasso Science Institute, L'Aquila, Italy \newline and DISIM, Department of Information Engineering, Computer Science and Mathematics \\ University of L'Aquila, Italy}\email{delyan.zhelyazov@gssi.it }

\newtheorem{theorem}{Theorem}
\newtheorem{lemma}[theorem]{Lemma}

\newtheorem{proposition}[theorem]{Proposition}

\theoremstyle{remark}

\begin{document}


\maketitle

\begin{abstract}
In this paper we perform the analysis of spectral properties of the linearized system around constant states and dispersive shock for a 1-D compressible Euler system with dissipation--dispersion terms. The dispersive term is originated by the quantum effects described through the Bohm potential, as customary in Quantum Hydrodynamic models. The analysis performed in this paper includes the computation of the linearized operator and the spectral stability through the Evans function method.
\end{abstract}

%

\section{Introduction}
The aim of this paper is to   study the following Euler system with dissipation--dispersion terms: 
\begin{equation}
\label{eq_sys}
\left\{
\begin{array}{ll}
\rho_t+m_x=0,\\
m_t+\Big{(}\frac{m^2}{\rho}+\rho^\gamma\Big{)}_x=\epsilon \mu m_{xx}+\epsilon^2 k^2 \rho \Big(\frac{(\sqrt{\rho})_{xx}}{\sqrt{\rho}}\Big{)}_x,
\end{array}
\right.
\end{equation}
where $t\geq0$, $x \in \mathbb{R}$, $\rho=\rho(t,x)>0$, $m=m(t,x)$, $\gamma \geq 1$, $0<\epsilon\ll1$, $\mu>0$, $k>0$. The positive coefficients $\epsilon \mu$ and $\epsilon^2 k^2$ stand   for the viscosity and dispersive coefficients, respectively, and $\rho^\gamma$ is the pressure.
The particular shape of the dispersion terms is due to the Bohm potential, and the resulting system is referred to as the quantum hydrodynamics system, being  used for instance in superfluidity or to model semiconductor devices. 
We are in particular interested in the analysis of the linearized version of \eqref{eq_sys} around special solutions, as constant states and   \emph{dispersive shocks}, namely solutions of  \eqref{eq_sys} written as travelling waves
\begin{equation*}
\rho=P\Big{(}\frac{x-s t}{\epsilon}\Big{)}\mbox{, }m=J\Big{(}\frac{x-s t}{\epsilon}\Big{)}\mbox{. }
\end{equation*}
As customary,  the speed  $s \in \mathbb{R}$  of the travelling wave and its limiting end states 
\begin{equation*}
\lim_{y\rightarrow \pm \infty}P(y)=P^{\pm}\  \hbox{and}\ \lim_{y\rightarrow \pm \infty}J(y)=J^{\pm}
\end{equation*}
are assumed to verify the Rankine--Hugoniot conditions:
\begin{align}
\label{Rankine_Hugoniot}
J^+-J^-&=s(P^+-P^-),\\
\label{Rankine_Hugoniot2}
\Big{(}\frac{J^2}{P}+P^{\gamma}\Big{)}^+-\Big{(}\frac{J^2}{P}+P^{\gamma}\Big{)}^-&=s(J^+-J^-).
\end{align}
The existence of such solution is studied in full details in the companion paper \cite{Zhelyazov}, where  in particular the interplay between the diffusive and dispersive effects is analyzed  in connections with the existence, monotonicity, and stability of such solutions.  The notion and study of the effect of dispersive terms have been first considered 
by \cite{Gurevich, Sagdeev}, see also \cite{Gurevich1, Hoefer, Nov, Zak}, while a fairly complete analysis, via the Whitham modulation theory, has been investigated in \cite{Hoefer1}, which also includes a wide bibliography on these topics.  The first attempt to analyze the spectral theory of the linearized operator around dispersive shocks has been discussed in  \cite{Humpherys} regarding the case of  $p$-system with real viscosity and linear capillarity, but only in the case of monotone shocks, while the mathematical theory  of the 
quantum hydrodynamic  systems has been developed in \cite{AM1, AM2, AMtf, AMDCDS, AS, Michele1, Michele, DFM, DM1, DM,  GLT}.

  In the present paper,  we shall set up the dynamical systems solved by these profiles, obtained as heteroclinic connections, and give some numerical computations of them in the next two sections.
Moreover  Section \ref{section_linerazation} is devoted to  the study of the linearization of \eqref{eq_sys}, both around a profile and around a constant state, the latter case being relevant for the subsequent analysis, which strongly relies on the linearized, constant coefficient, asymptotic operators, obtained as the parameter of the profile tends to $\pm\infty$.
 In Section \ref{section_point_spectrum} we treat the analysis of the spectrum of linearized operators, in particular giving a resolvent estimate for the constant coefficient case, namely, for the linearized operator about a constant state. To investigate the point spectrum of the linearized operator around a profile, one needs to supplement the  resolvent estimate (available for large eigenvalues) with the study of the Evans function (see, for instance \cite{Sandstede}); the corresponding numerical results about it are collected in Section \ref{section_the_evans_function}.

\section{The equations for the profile}\label{section_2D_system}
In order to analyze  travelling wave profiles, we plug 
\begin{equation*}
\rho=P\Big{(}\frac{x-s t}{\epsilon}\Big{)}\mbox{, }m=J\Big{(}\frac{x-s t}{\epsilon}\Big{)}
\end{equation*}
in \eqref{eq_sys}, and
we rewrite the Bohm potential in   conservative form as follows:
\begin{equation*}
\rho \Big(\frac{(\sqrt{\rho})_{xx}}{\sqrt{\rho}}\Big{)}_x=\frac{1}{2}\Big{(}\rho (\ln \rho)_{xx}\Big{)}_x\mbox{. }
\end{equation*}
After substituting the profiles $P$ and $J$ in the system \eqref{eq_sys} and multiplying by $\epsilon$ we obtain
\begin{align}
- s P'+J'&=0,\label{preq1}\\
-s J' + \Big{(}\frac{J^2}{P}+P^\gamma\Big{)}'&=\mu J''+\frac{k^2}{2}(P(\ln P)'')'\mbox{, }
\label{preq2}
\end{align}
where $'$ denotes $d/dy$ and $P=P(y)$, $J=J(y)$. 
Integrating equation \eqref{preq1}, we get 
\begin{equation*}
-s \int_{-\infty}^y P'(x) dx + \int_{-\infty}^y J'(x)dx=0,
\end{equation*}
that is
\begin{equation*}
J(y)-sP(y)=J^--sP^-.
\end{equation*}
We can also integrate \eqref{preq1} from $y$ to $+\infty$ to get
\begin{equation*}
J(y)-sP(y)=J^+-sP^+,
\end{equation*}
that is
\begin{equation}
\label{equation_j}
J(y)=sP(y)-A,
\end{equation}
 with
\begin{equation*}
A=s P^{\pm}-J^{\pm},
\end{equation*}
thanks to the Rankine-Hugoniot condition \eqref{Rankine_Hugoniot}.\\
Similarly, we can integrate \eqref{preq2} up to $\pm \infty$, after substituting the expression for $J(y)$ into it, to obtain 
 the planar ODE
\begin{equation}\label{2Dsys}
P''=\frac{2}{k^2} f(P) - \frac{2 s \mu}{k^2} P' + \frac{P'^2}{P}.
\end{equation}
In \eqref{2Dsys},
\begin{align}
f(P)&=-s(s P-A)+\frac{(sP-A)^2}{P}+P^{\gamma}-B\nonumber\\
&=P^{\gamma}-(As+B)+\frac{A^2}{P}
\label{fun_f}
\end{align}
and the constant $B$ is given by 
\begin{equation*}
B=-s J^{\pm}+\Big{(}\frac{J^2}{P}+P^{\gamma}\Big{)}^{\pm},
\end{equation*}
as follows from \eqref{Rankine_Hugoniot2}.

Introducing the new variable $P'=Q$, we rewrite \eqref{2Dsys} as a first-order system of ODEs
\begin{align}
\label{sys_ODE1}
P'=Q&=f_1,\\
Q'=\frac{2 f(P)}{k^2}-\frac{2\mu s}{k^2} Q+\frac{Q^2}{P}&=f_2.\label{sys2}
\end{align}
Finally, it is worth to observe that the constants $A,B$ in $f(P)$ can be expressed solely in terms of $P^{\pm}$ as follows:
\begin{align*}
f(P)&= P^\gamma +\frac{P^-P^+}{P}\frac{(P^+)^\gamma-(P^-)^\gamma}{P^+-P^-}\nonumber\\
&\ -\frac{(P^+)^{\gamma+1}-(P^-)^{\gamma+1}}{P^+-P^-}.
\end{align*}
\section{Numerical computation of the profiles}
\label{section_numerical_profiles}
The existence of profiles under different assumptions on the asymptotic states is proved in \cite{Zhelyazov}; here we collect some numerical computations about such profiles in the parameters' range where the existence is proved. 
In what follows, we briefly recall     the case of large shocks, which in particular includes the possibility of non monotone profiles.\\
To this aim, let us denote
\begin{equation*}
F(P)= \frac{P^{\gamma+1}}{\gamma+1}-(As+B)P+A^2 \ln P,
\end{equation*}
and 
\begin{align*}
&F_1(P)=F(P)-\Big{(}\frac{s \mu P}{k}\Big{)}^2-F(P^-),\\
&\tilde{F}_1(\xi)=F(\xi)-\Big{(}\frac{s \mu \xi}{k}\Big{)}^2-F(P^+)
\end{align*}
Then the next theorem holds.
\begin{theorem}\label{lemma_global_existence} Let $s>0$ and $P^+<P^-$. If $F_1(P)\leq 0$ for $0<P<P^+$, then there is a traveling wave profile, connecting $[P^-,0]$ to $[P^+,0]$. If in addition
\begin{equation*}
\frac{s \mu}{k}<\sqrt{-2 f'(P^+)},
\end{equation*}
then the traveling wave profile is non-monotone.\\
Let $s<0$ and $P^-<P^+$. If $\tilde{F}_1(\xi)\leq 0$ for $0<\xi<P^-$, then there is a traveling wave profile, connecting $[P^-,0]$ to $[P^+,0]$. If in addition
\begin{equation*}
-\frac{s \mu}{k}<\sqrt{-2 f'(P^-)},
\end{equation*}
then the traveling wave profile is non-monotone.
\end{theorem}
The condition $s>0$, $P^+<P^-$ is implied by the condition for a 2-shock with a supersonic right state, and $u^+>c_s(P^+)$, where $u^+=J^+/P^+$, and $c_s(P^+)$ is the speed of sound.\\
We describe the result of Theorem \ref{lemma_global_existence} in the particular case of $\gamma = 1$, in Figure \ref{figure_profiles} (a), and $\gamma = \frac{3}{2}$, in Figure \ref{figure_profiles} (b).
The condition of Theorem \ref{lemma_global_existence} can be easily checked numerically for the profiles presented in Figure \ref{figure_profiles}. 
To give a numerical computation for such profile, we analyze the jacobian of \eqref{sys_ODE1}-\eqref{sys2} at the steady-state $[P^-,0]$. 
With the notation $a=\frac{2}{k^2}>0$, $b=\frac{2 s \mu}{k^2}>0$,  the jacobian is given by
\begin{equation*}
     J=  \begin{bmatrix}
    0 & a \\
    f'(P^-) & -b
    \end{bmatrix}.
\end{equation*}
a direct computation shows that his eigenvalues are
\begin{equation*}
\mu_{1,2}(P^-) = (-b \pm \sqrt{b^2+4 a f'(P^\pm)})/2
\end{equation*}
and, since $f'(P^-)>0$, we have $\sqrt{b^2+4 a f'(P^-)}>b$, and we conclude $\mu_1(P^-) < 0 < \mu_2(P^-)$ and therefore  the steady-state $[P^-,0]$ is a saddle. 
The eigenvector corresponding to the positive eigenvalue $\mu_2(P^-)$ with negative second component is
\begin{equation*}
v_2=  -\begin{bmatrix}
    \frac{(b + \sqrt{b^2+4 a f'(P^-)})}{2 f'(P^-)}\\
    1
    \end{bmatrix}
\end{equation*}
and, as customary, it is tangent to the unstable subspace of the equilibrium $[P^-,0]$. 
Then, to numerically compute the traveling wave profile, we numerically integrate \eqref{sys_ODE1}-\eqref{sys2} 
with the initial condition $[P^-,0]+h v_2/|v_2|$, with $h>0$, sufficiently small, for instance $h=10^{-5}$; see Figure \ref{figure_profiles}. The global existence Lemma from \cite{Zhelyazov} applies to the parameters of these profiles. The jacobian at $[P^+,0]$ for the parameters in Figure \ref{figure_profiles} (b) has imaginary eigenvalues, which shows that the profile is non-monotone.\\
Using  $P(y)$, we obtain $J(y)$ from equation \eqref{equation_j}.
\begin{figure}[H]
\centering
\begin{subfigure}[b]{0.45\textwidth}
\includegraphics[scale=0.6]{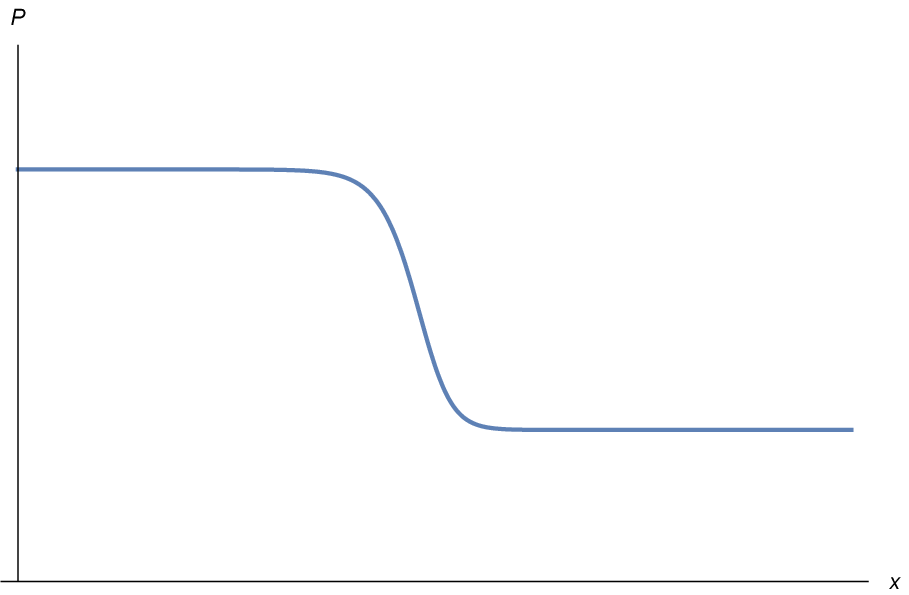}
\caption{}
\label{fig:fig_monotone}
\end{subfigure}%
\begin{subfigure}[b]{0.45\textwidth}
\includegraphics[scale=0.6]{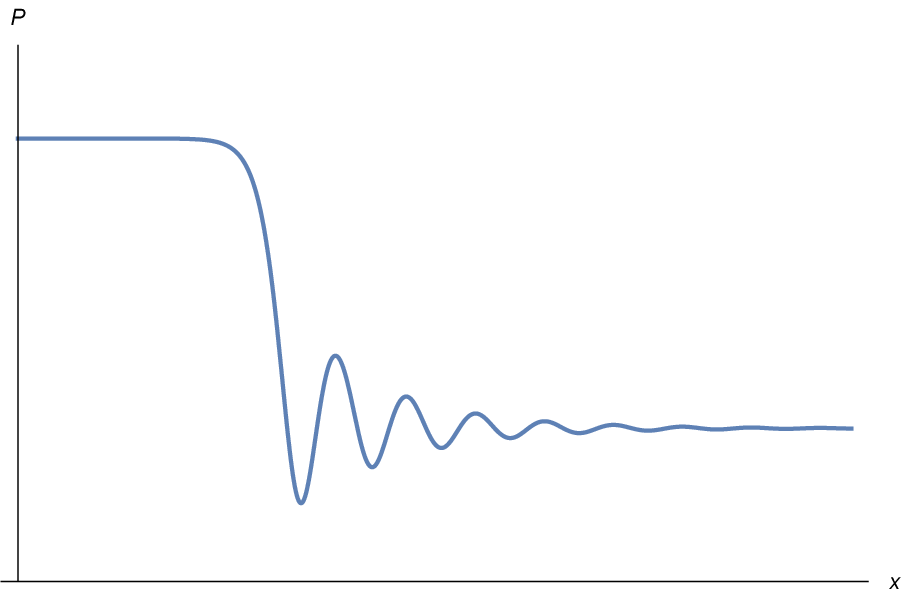}
\caption{}
\label{fig:fig_oscillatory}
\end{subfigure}

\caption{(a) A monotone profile for parameters $P^+=5.69$, $P^-=8.61$, $s=1$, $\gamma=1$, $\mu=8$, $k=1$. (b) An oscillatory profile for parameters $P^+=3.5$, $P^-=4.63$, $s=1$, $\gamma=\frac{3}{2}$, $\mu=0.25$, $k=\sqrt{2}$.}
\label{figure_profiles}
\end{figure}

\section{Linearization}
\label{section_linerazation}
Using the change of variables $\tau=t/\epsilon$, $y=(x-st)/\epsilon$, we get the full linearized operator around the profile for \eqref{eq_sys}:
\begin{equation}
\label{operator_L}
L
\begin{bmatrix}
\tilde{\rho}\\
\tilde{J}
\end{bmatrix}
=\begin{bmatrix}
    s \tilde{\rho}_y - \tilde{J}_y\\
    s \tilde{J}_y +(\frac{J^2}{P^2}\tilde{\rho})_y-(\frac{2 J}{P}\tilde{J})_y-\gamma (P^{\gamma-1}\tilde{\rho})_y+\mu \tilde{J}_{yy}+L_V\tilde{\rho}
    \end{bmatrix},
\end{equation}
where
\begin{equation*}
L_V \tilde{\rho} = \frac{k^2}{2}\tilde{\rho}_{yyy}-2 k^2 \Big{(} (\sqrt{P})_y\Big{(}\frac{\tilde{\rho}}{\sqrt{P}}\Big{)}_y\Big{)}_y,
\end{equation*}
with associated eigenvalue problem given by
\begin{equation}
\label{eq_variable_coeff}
\lambda \begin{bmatrix}
\tilde{\rho}\\
\tilde{J}
\end{bmatrix} = L\begin{bmatrix}
\tilde{\rho}\\
\tilde{J}
\end{bmatrix}.
\end{equation}
For the analysis of the eigenvalue problem \eqref{eq_variable_coeff} we will need the Evans function and to this end it is also important to re--express the above linearized systems in terms of integrated variables, 
because  this transformation removes the zero eigenvalue (always present, being its eigenfunction given by the derivative of the profile), without further modifications of the spectrum; see, for instance \cite{Humpherys}.
To this end, consider
\begin{equation*}
\hat{\rho}(x)=\int_{-\infty}^x \tilde{\rho}(y)dy,\mbox{ }\hat{J}(x)=\int_{-\infty}^x \tilde{J}(y) dy,
\end{equation*} 
Integrating the equation \eqref{eq_variable_coeff} it follows that for $\lambda \neq 0$ the integrated variables $\tilde{\rho}$ and $\tilde{J}$ decay exponentially as $|x|\rightarrow +\infty$. Expressing $\tilde{\rho}$ and $\tilde{J}$ in terms of $\hat{\rho}$ and $\hat{J}$, and integrating \eqref{eq_variable_coeff} from $-\infty$ to $x$ we get the system in integrated variables:
\begin{align}
\label{sys_integrated_variables}
\lambda \hat{\rho} &= s \hat{\rho}'-\hat{J}',\\
\label{sys_integrated_variables1}
\lambda \hat{J} &= f_1 \hat{\rho}' + f_2 \hat{J}'+\mu \hat{J}''+\frac{k^2}{2}\hat{\rho}'''-2 k^2 (\sqrt{P})'\Big{(}\frac{\hat{\rho}'}{\sqrt{P}}\Big{)}',
\end{align}
with
\begin{align*}
f_1(x) &= \frac{J(x)^2}{P(x)^2}-\gamma P(x)^{\gamma-1},\\
f_2(x) &= s - 2 \frac{J(x)}{P(x)}.
\end{align*}
We can rewrite \eqref{sys_integrated_variables}-\eqref{sys_integrated_variables1} as $V'=\hat{M}(x,\lambda)V$, where $V=[\hat{\rho},\hat{J},u_1,u_2]^T$ and 
\begin{equation}
\label{mat_M1}
\hat M(x,\lambda) = \begin{bmatrix}
0 & 0 & 1 & 0\\
-\lambda & 0 & s & 0\\
0 & 0 & 0 & 1\\
\frac{2 \lambda f_2}{k^2} & \frac{2 \lambda}{k^2} & \frac{2 \lambda \mu}{k^2}-\frac{2 f_1}{k^2}-\frac{2 s f_2}{k^2}-\frac{(P')^2}{P^2}& \frac{2 P'}{P}-\frac{2 s \mu}{k^2}
\end{bmatrix}.
\end{equation}
The limit of $\hat{M}(x,\lambda)$ as $x \rightarrow \pm \infty$ is given by
\begin{equation}
\label{mat_M}
M^{\pm} = \begin{bmatrix}
0 & 0 & 1 & 0\\
-\lambda & 0 & s & 0\\
0 & 0 & 0 & 1\\
\frac{2 \beta^\pm \lambda}{k^2} & \frac{2 \lambda}{k^2} & \frac{2}{k^2}(\mu \lambda - s \beta^\pm - \alpha^\pm) & -\frac{2 s \mu}{k^2}
\end{bmatrix},
\end{equation}
which is the same asymptotics one shall obtain from the original linearization \eqref{operator_L} after rewriting the latter as a first order $4\times4$ system. This will lead to a 
 linear, constant coefficient operator which is crucial in the spectral analysis of the linerazed operator around the profile.
 
A related constant coefficient linear operator is clearly obtained linearizing our original system about a constant state, thus obtaining the same operator, but for $s=0$.
 Denote
\begin{equation}\label{eq:constalbeta}
\alpha=\frac{\bar J^2}{\bar P^2}-\gamma \bar P^{\gamma-1}\mbox{ , }\beta= -\frac{2 \bar J}{\bar P}.
\end{equation}
Then the operator, corresponding to the linearization around the constant steady--state $(\bar P, \bar J)$ is
\begin{eqnarray*}
L_c
\begin{bmatrix}
\tilde{\rho}\\
\tilde{J}
\end{bmatrix}
=\begin{bmatrix}
     - \tilde{J}'\\
    \alpha \tilde{\rho}'+\beta \tilde{J}'+\mu \tilde{J}''+\frac{k^2}{2}\tilde{\rho}'''
    \end{bmatrix},
\end{eqnarray*}
where $'=d/dy$.
The asymptotic operators  at $\pm\infty$ for \eqref{operator_L}  are given by
\begin{eqnarray*}
L_{\pm\infty}
\begin{bmatrix}
\tilde{\rho}\\
\tilde{J}
\end{bmatrix}
=\begin{bmatrix}
    s \tilde{\rho}' - \tilde{J}'\\
    \alpha^\pm \tilde{\rho}'+\beta^\pm \tilde{J}'+\mu \tilde{J}''+\frac{k^2}{2}\tilde{\rho}'''
    \end{bmatrix}
\end{eqnarray*}
where 
\begin{equation*}
\alpha^{\pm}=\frac{(J^{\pm})^2}{(P^{\pm})^2}-\gamma (P^\pm)^{\gamma-1}\mbox{ , }\beta^\pm=s-\frac{2 (J^\pm)}{P^\pm}.
\end{equation*}
We may rewrite the equation
\begin{equation*}
\lambda \begin{bmatrix}
\tilde{\rho}\\
\tilde{J}
\end{bmatrix} = L_{\pm\infty}\begin{bmatrix}
\tilde{\rho}\\
\tilde{J}
\end{bmatrix}
\end{equation*}
as a first order system $V'=M^{\pm}V$, with $V=[\tilde{\rho},\tilde{J},u_1,u_2]^T$, and the matrices $M^{\pm}$ are exactly \eqref{mat_M}.

\section{Spectral analysis}\label{section_point_spectrum} 
In this section we shall locate the spectrum of our linearized operators,  starting from the constant steady--case. Indeed, there are no eigenfunctions   which decay at $\pm\infty$ for the constant coefficient linear operator $L_c$, thus the spectrum of the latter reduces to the essential one.  In addition, we shall perform a resolvent estimate in that case  valid in the whole unstable half plane $\Re(\lambda)>0$, to present  in particular a simple calculation which is useful to understand the behavior of the spectrum for the linearized operator about the profile $L$; more details in this direction are presented in \cite{Zhelyazov}.

\subsection{Constant steady state}
To obtain a resolvent estimate in $H^1$, we need to invert in that space  the relation
 \begin{equation*}
(\lambda I - L_c) \begin{bmatrix}
\tilde{\rho}\\
\tilde{J}
\end{bmatrix} = \begin{bmatrix}
f_1\\
f_2
\end{bmatrix}
\end{equation*}
and prove $(\lambda I - L_c)^{-1}: H^1 \to H^1$ is bounded for $\Re(\lambda)>0$. 
With the notation 
\eqref{eq:constalbeta}, the system at hand rewrites
\begin{align}
\lambda \tilde{\rho}&=- \tilde{J}'+f_1,\label{nonh_eq1}\\
\lambda \tilde{J}&=\alpha \tilde{\rho}'+\beta \tilde{J}'+\mu \tilde{J}''+\frac{k^2}{2}\tilde{\rho}'''+f_2 \label{nonh_eq2}
\end{align}
and denote $f=[f_1,f_2]^T$.
%

The speed of sound is $c_s(\bar P)=\sqrt{\gamma \bar P^{\gamma-1}}$ and the flow velocity is $\bar u=\bar J/\bar P$. The condition for a subsonic steady-state is $|\bar u|<c_s(\bar P)$, which becomes
$\alpha<0$ after squaring. As an example, we present here below a resolvent estimate  in the above framework in $H^1$.
\begin{proposition}
If the steady-state is subsonic, then for $\Re(\lambda)>0$ we have the resolvent estimate
\begin{equation*}
\Vert [\tilde{\rho},\tilde{J}]^T\Vert_{H^1}\leq C \frac{ \Vert f \Vert_{H^1}}{h(\Re(\lambda))},
\end{equation*}
where $C=\max\{\frac{k^2}{2},-\alpha,1\}$ and $h(\Re(\lambda)):=\min\{-\alpha \Re(\lambda),\Re(\lambda),\frac{k^2}{2}\Re(\lambda),\mu\}$.
\end{proposition}
\begin{proof}
 We multiply equation \eqref{nonh_eq1} by $-\alpha \overline{\tilde{\rho}}$ and equation \eqref{nonh_eq2} by $\overline{\tilde{J}}$. After taking the real part we get
\begin{align*}
\Re(\lambda)(-\alpha |\tilde{\rho}|^2+|\tilde{J}|^2)&=\alpha \Re(\tilde{J}'\overline{\tilde{\rho}}+\tilde{\rho}' \overline{\tilde{J}})+\beta \Re(\tilde{J}'\overline{\tilde{J}})
+\mu \Re({\tilde{J}''\overline{\tilde{J}}})\\
&+\frac{k^2}{2}\Re(\tilde{\rho}'''\overline{\tilde{J}})+\Re(-\alpha f_1 \overline{\tilde{\rho}}+f_2\overline{\tilde{J}}).
\end{align*}
Let $\tilde{\rho}=\tilde{\rho_r}+i \tilde{\rho_i}$ and $\tilde{J}=\tilde{J_r}+i \tilde{J_i}$. We have
\begin{align*}
\alpha \Re(\tilde{J}'\overline{\tilde{\rho}}+\tilde{\rho}' \overline{\tilde{J}})&= \alpha (\tilde{\rho_r} \tilde{J_r} + \tilde{\rho_i} \tilde{J_i})',\\
\beta \Re(\tilde{J}'\overline{\tilde{J}}) &= \frac{\beta}{2}(|\tilde{J}|^2)',\\
\mu \Re({\tilde{J}''\overline{\tilde{J}}}) &= \mu(\tilde{J_r}''\tilde{J_r}+\tilde{J_i}''\tilde{J_i}).
\end{align*}
Therefore
\begin{align}
\Re(\lambda)(-\alpha |\tilde{\rho}|^2+|\tilde{J}|^2)&=\Big{(}\alpha (\tilde{\rho_r} \tilde{J_r} + \tilde{\rho_i} \tilde{J_i})+\frac{\beta}{2}|\tilde{J}|^2\Big{)}'\nonumber\\
&+\mu(\tilde{J_r}''\tilde{J_r}+\tilde{J_i}''\tilde{J_i})+\frac{k^2}{2}\Re(\tilde{\rho}'''\overline{\tilde{J}})+\Re(-\alpha f_1 \overline{\tilde{\rho}}+f_2\overline{\tilde{J}})\label{eq_const2}.
\end{align}
After integration, the derivative term on the right hand side of \eqref{eq_const2} disappears and we get
\begin{align}
\Re(\lambda)(-\alpha\int |\tilde{\rho}|^2 dy +\int |\tilde{J}|^2dy)&=\int \mu(\tilde{J_r}''\tilde{J_r}+\tilde{J_i}''\tilde{J_i})dy +\int \frac{k^2}{2}\Re(\tilde{\rho}'''\overline{\tilde{J}}) dy\nonumber\\
&
\label{eq_const3}+\int \Re(-\alpha f_1 \overline{\tilde{\rho}}+f_2\overline{\tilde{J}}) dy.
\end{align}
By integration by parts we obtain 
\begin{equation}
\label{eq_const4}
\int \mu(\tilde{J_r}''\tilde{J_r}+\tilde{J_i}''\tilde{J_i})dy = - \mu \int ((\tilde{J_r}')^2+(\tilde{J_i}')^2)dy=-\mu \int |\tilde{J}'|^2 dy.
\end{equation}
We have 
\begin{equation*}
\Re(\tilde{\rho}'''\overline{J})=\tilde{\rho_r}'''\tilde{J_r}+\tilde{\rho_i}'''\tilde{J_i},
\end{equation*}
hence by integration by parts
\begin{equation*}
\int \frac{k^2}{2}\Re(\tilde{\rho}'''\overline{\tilde{J}}) dy=-\frac{k^2}{2}\int (\tilde{J_r}' \tilde{\rho_r}''+\tilde{J_i}' \tilde{\rho_i}'')dy.
\end{equation*}
Let $f_1=f_{1,r}+i f_{1,i}$.
Substituting $\tilde{J_r}'$ and $\tilde{J_i}'$ from equation  \eqref{nonh_eq1} yields
\begin{align*}
-\frac{k^2}{2}\int (\tilde{J_r}' \tilde{\rho_r}''+\tilde{J_i}' \tilde{\rho_i}'')dy = \\
-\frac{k^2}{2}\int \Big{(} (f_{1,r}-\Re (\lambda) \tilde{\rho_r}+\Im(\lambda) \tilde{\rho_i})\tilde{\rho_r}''+(f_{1,i}-\Re (\lambda) \tilde{\rho_i}-\Im(\lambda) \tilde{\rho_r})\tilde{\rho_i}''\Big{)}dy\\
=-\frac{k^2}{2}\int (t_1+t_2+t_3)dy,
\end{align*}
where
\begin{align*}
t_1 &= \Im(\lambda) (\tilde{\rho_i} \tilde{\rho_r}''-\tilde{\rho_i}'' \tilde{\rho_r}),\\
t_2 &= -\Re(\lambda) (\tilde{\rho_r}\tilde{\rho_r}''+\tilde{\rho_i}\tilde{\rho_i}''),\\
t_3 & = f_{1,r} \tilde{\rho_r}''+f_{1,i} \tilde{\rho_i}''.
\end{align*}
Using integration by parts we get
\begin{equation*}
-\frac{k^2}{2}\int t_1 dy = -\frac{k^2 \Im(\lambda)}{2}\int (-\tilde{\rho_r}'\tilde{\rho_i}'+\tilde{\rho_i}'\tilde{\rho_r}')dy=0,
\end{equation*}
and again by integration by parts
\begin{equation*}
-\frac{k^2}{2}\int t_2 dy = -\frac{\Re(\lambda) k^2}{2}\int \Big{(} (\tilde{\rho_r}')^2+(\tilde{\rho_i}')^2\Big{)} dy=-\frac{\Re(\lambda) k^2}{2}\int |\tilde{\rho}'|^2 dy,
\end{equation*}
which yields 
\begin{equation}
\label{eq_const5}
\int \frac{k^2}{2}\Re(\tilde{\rho}'''\overline{\tilde{J}}) dy = -\frac{\Re(\lambda) k^2}{2}\int |\tilde{\rho}'|^2 dy-\frac{k^2}{2}\int t_3 dy.
\end{equation}
Substituting \eqref{eq_const4} and \eqref{eq_const5} into \eqref{eq_const3} we get
\begin{align}
&\Re(\lambda)\Big{(}-\alpha\int |\tilde{\rho}|^2 dy +\int |\tilde{J}|^2dy+\frac{ k^2}{2}\int |\tilde{\rho}'|^2 dy\Big{)}+\mu \int |\tilde{J}'|^2 dy \nonumber\\
\label{eq_const6}
&=\int \Re(-\alpha f_1 \overline{\tilde{\rho}}+f_2\overline{\tilde{J}}) dy-\frac{k^2}{2}\int t_3 dy.
\end{align}
Using integration by parts we get
\begin{equation}
\label{eq_intt3}
-\frac{k^2}{2}\int t_3 dy=-\frac{k^2}{2}\int (f_{1,r} \tilde{\rho_r}''+f_{1,i} \tilde{\rho_i}'') dy
=\frac{k^2}{2}\int (f_{1,r}'\tilde{\rho_r}'+f_{1,i}'\tilde{\rho_i}')dy.
\end{equation}
Equation \eqref{eq_intt3} yields
\begin{align}
&-\frac{k^2}{2}\int t_3 dy \leq \frac{k^2}{2} \int \sqrt{(f_{1,r}')^2+(f_{1,i}')^2}\sqrt{(\tilde{\rho_r}')^2+(\tilde{\rho_i}')^2}dy\nonumber\\
\label{ineq_intt3}
&=\frac{k^2}{2}\int|f_1'||\tilde{\rho}'|dy.
\end{align}
With the notation $h(\Re(\lambda)):=\min\{-\alpha \Re(\lambda),\Re(\lambda),\frac{k^2}{2}\Re(\lambda),\mu\}$, we estimate the left-hand side of \eqref{eq_const6} by
\begin{align}
&h(\Re(\lambda))\Big (\int |\tilde{\rho}|^2dy+\int|\tilde{J}|^2dy+\int |\tilde{\rho}'|^2 dy+\int |\tilde{J}'|^2 dy \Big ) \nonumber\\
\label{estimate_lhs}
&\leq \Re(\lambda)\Big{(}-\alpha\int |\tilde{\rho}|^2 dy +\int |\tilde{J}|^2dy+\frac{ k^2}{2}\int |\tilde{\rho}'|^2 dy\Big{)}+\mu \int |\tilde{J}'|^2 dy.
\end{align}
Let $C=\max\{\frac{k^2}{2},-\alpha,1\}$. Taking into account \eqref{ineq_intt3} we estimate the right-hand side of \eqref{eq_const6} by 
\begin{align}
&\int \Re(-\alpha f_1 \overline{\tilde{\rho}}+f_2\overline{\tilde{J}}) dy-\frac{k^2}{2}\int t_3 dy\nonumber\\
&\leq \frac{k^2}{2}\int|f_1'||\tilde{\rho}'|dy + \int (|\alpha| |f_1| |\tilde{\rho}|+|f_2| |\tilde{J}|) dy \nonumber\\
&\leq C(\int|f_1'||\tilde{\rho}'|dy+\int |f_1| |\tilde{\rho}|dy+\int |f_2| |\tilde{J}|dy) \nonumber\\
&\leq C(\int|f_1'||\tilde{\rho}'|dy+\int |f_1| |\tilde{\rho}|dy+\int |f_2| |\tilde{J}|dy+ \int |f_2'| |\tilde{J'}|dy) \nonumber\\
& \label{estimate_rhs} \leq C \Vert f \Vert_{H^1} \Vert [\tilde{\rho},\tilde{J}]^T\Vert_{H^1}.
\end{align}
From inequalities \eqref{estimate_lhs}  and \eqref{estimate_rhs} we get
\begin{equation*}
h(\Re(\lambda))\Vert [\tilde{\rho},\tilde{J}]^T\Vert_{H^1}^2\leq C \Vert f \Vert_{H^1} \Vert [\tilde{\rho},\tilde{J}]^T\Vert_{H^1}.
\end{equation*}
This yields the proof of the proposition.
\end{proof}

For completeness, we underline that the stability of the essential spectrum can be obtained  by Lemma \ref{lem:essspec} below, which is a restatement for the present case involving $L_c$ of the corresponding result in  \cite{Zhelyazov} for the asymptotic operators $L_{\pm\infty}$. The statement refers to the asymptotic operators at 
$x\to\pm\infty$, as it is used to locate the essential spectrum also for the linearized operator around the profile \cite{Kapitula},
 but clearly can be applied to any linearized operator about a constant state.
To this end, we report the dispersion relation referred to $L_c$ here below:
\begin{equation*}
\lambda^2+\xi(\mu \xi - i \beta)\lambda+\xi^2\Big{(}-\alpha+\frac{k^2 \xi^2}{2}\Big{)}=0.
\end{equation*}
\begin{lemma}\label{lem:essspec}
If $\alpha \leq 0$, then the essential spectrum is in the closed left half-plane. Also, if $\xi \neq 0$, then $\Re(\lambda_{1,2})<0$.
\end{lemma}

In contrast with the situation described in the proposition above, the localization of the point spectrum of the linearized operator along a profile is more involved, and in particular the resolvent estimate obtained above is proved in the companion paper \cite{Zhelyazov} only for $\Re(\lambda)$ sufficiently big.
Thus, to locate the point spectrum in that case, an efficient method is to locate   the zeros of the Evans function,  the latter being exactly the eigenvalues of the operator under consideration. The argument needed requires a careful analysis of the behavior of such function  for large $|\lambda|$, which gives a quantitative version of  the asymptotic results of \cite{Sandstede}, excluding the presence of eigenvalues for $|\lambda|> C$ with an \emph{explicit} bound for the constant.   This result is completed with   the numerical study of the winding number of the Evans function in a sufficiently big contour; all details can be found in \cite{Zhelyazov}. In the next section we shall briefly recall the main ingredients/results for the latter numerical analysis, thus yielding to a numerical evidence of point spectrum stability.

\subsection{The Evans function and its numerical evaluation}
\label{section_the_evans_function}
To define  the Evans function, let us consider the equation $Y'=\hat{M}(y,\lambda)Y$, where $\hat{M}(y,\lambda)$ is defined in \eqref{mat_M1}.
As it is manifest, its limits at $\pm\infty$ are given by  the matrices $M^{\pm}$, defined by \eqref{mat_M}, corresponding to limit states $P^{\pm}$, and we assume  these matrices are hyperbolic. This is always true if we are to the right of the bound for the essential spectrum. In addition, we assume that $M^{-}$ has $k$ unstable eigenvalues $\nu^{-}_1,\dots,\nu^{-}_k$ (i.e.\ $\Re(\nu^{-}_i)>0$), and $M^{+}$ has $n-k$ stable eigenvalues $\nu^{+}_1,...,\nu^{+}_{n-k}$ (i.e.\ $\Re(\nu^{+}_i)<0$), and denote the corresponding (normalized) eigenvectors by $v^{\pm}_i$. In our case $n=4$ and $k=2$. Let $Y^{-}_i$ be a solution of $Y'=M(y,\lambda)Y$, satisfying $exp(\nu^- y)Y^{-}(y)$ tends to $v^{-}_i$ as $y \rightarrow -\infty$ and $exp(\nu^+ y)Y^{+}(y)$ tends to $v^{+}_i$ as $y \rightarrow +\infty$. Then,  the Evans function can be defined by
\begin{equation*}
E(\lambda) = det(Y^-_1(0), .., Y^-_k(0),  Y^+_1(0), ... ,Y^+_{n-k}(0)).
\end{equation*}
As a consequence, a point $\lambda \in \mathbb{C}$ is in the point spectrum of $L$ if and only if $E(\lambda)=0$.
In contrast, to compute the Evans function numerically, we use the compound matrix method; for instance, see \cite{Humpherys} .
This method  is used in order to get a stable numerical procedure, because the system $Y'=M(y,\lambda)Y$ is numerically stiff.
Specifically, the compound matrix $B(y,\lambda)$ is given by:\
\begin{equation*}
B=\begin{bmatrix}
    m_{11}+m_{22} & m_{23} & m_{24} & -m_{13} & -m_{14} & 0 \\
    m_{32} & m_{11}+m_{33} & m_{34} & m_{12} & 0 & -m_{14} \\
    m_{42} & m_{43} & m_{11}+m_{44} & 0 & m_{12} & m_{13} \\
    -m_{31} & m_{21} & 0 & m_{22}+m_{33} & m_{34} & -m_{24} \\
    -m_{41} & 0 & m_{21} & m_{43} & m_{22}+m_{44} & m_{23} \\
     0 & -m_{41} & m_{31} & -m_{42} & m_{32} & m_{33}+m_{44}
  \end{bmatrix}.
\end{equation*}
We integrate the equation $\phi'=(B(y,\lambda)-\mu^-)\phi$ numerically  on a sufficiently large interval $[-L_1,0]$, where $\mu^-$ is the unstable eigenvalue of $B$ at $-\infty$ with maximal (positive) real part. Given a profile $[P(y),J(y)]$ at discrete points $(y_k)_{k=1}^N$, we obtain the matrix $B(y,\lambda)$, by linear interpolation.
Similarly we integrate the equation $\phi'=(B(y,\lambda)-\mu^+)\phi$ on $[0,L_1]$ backwards, where this time  $\mu^+$ is the stable eigenvalue of $B$ at $+\infty$ with minimal (negative) real part. Then, the coefficients $\mu^{\pm}$ compensate for the growth/decay at infinity. Finally, the Evans function can be constructed by means of linear combination of the components of the two solutions $\phi^\pm = (\phi^\pm_1, \dots, \phi^\pm_6)$ as follows:
\begin{equation*}
E(\lambda)=\phi^-_1\phi^+_6-\phi^-_2\phi^+_5+\phi^-_3\phi^+_4+\phi^-_4\phi^+_3-\phi^-_5\phi^+_2+\phi^-_6\phi^+_1\Big |_{y=0}.
\end{equation*}
Specifically, for our calculations we choose $L_1=40$.

Under the assumption that $E(\lambda)$ is analytic in the region surrounded by a closed contour $\Gamma$, and it does not vanish on the contour, we can use the winding number
\begin{equation}
\label{log_idnicator1}
\frac{1}{2 \pi i} \int_{\Gamma}\frac{E'(z)}{E(z)}dz
\end{equation}
to count the number of zeros inside the contour.  The remaining part of this paper is devoted to provide numerical evidence that the integral \eqref{log_idnicator1} is indeed zero in a sufficiently big contur $\Gamma$ lying in the unstable half plane.
We present the calculations on Section \ref{Numerics_E} in integrated variables.

\subsection{Numerical evaluation of $E(\lambda)$ for integrated variables}
\label{Numerics_E}
In this subsection we compute the Evans function for parameters $A=1$, $B=1.1$, $s=1$, $\gamma=3/2$, $\mu=1$, $k=\sqrt{2}$. These parameters in particular correspond to a non-monotone shock, and Theorem \ref{lemma_global_existence} applies to them. Here we use integrated variables to avoid the smallness of $E(\lambda)$ near zero, namely, we solve the ODEs
$\phi'=(\hat B(y,\lambda)-\mu^\pm)\phi$, where the compound matrix $\hat B$ is constructed from $\hat M$ defined in \eqref{mat_M1}.\\
In the present situation, the essential spectrum touches the origin $\lambda = 0$ and therefore, to rigorously use the Evans function tool in order to locate eigenvalues with small positive real part, one should first extend it across the essential spectrum. This rigorous analysis goes beyond the aim of the present discussion; however we numerically check the Evans function is well defined and different from zero at $\lambda =0$ by  showing that it is almost constant (and non zero) on a small semi-circular contour without a vertical segment in the unstable half-plane with radius $10^{-6}$ and with center at $\lambda=0$; see Figure \ref{figure_evans_function_small_contour}.

\begin{figure}[H]
\includegraphics[scale=0.6]{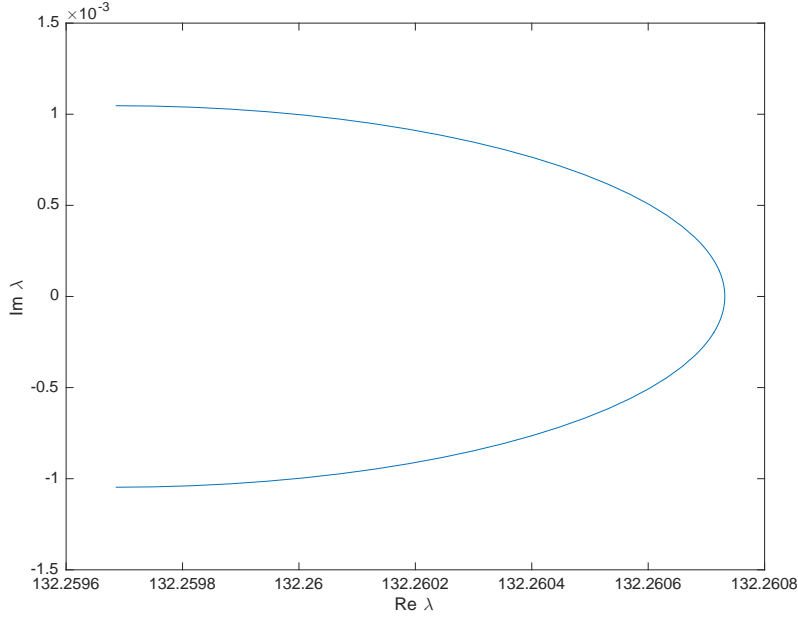}
\caption{The image of a small semi-circular contour without a vertical segment, with radius $10^{-6}$, and center at the origin through $E(\lambda)$.}
\label{figure_evans_function_small_contour}
\end{figure}

 To compute the initial conditions we integrate the reduced Kato ODE
\begin{equation*}
\frac{dr}{d\lambda}=\frac{dP}{d\lambda}r,
\end{equation*}
with the algorithm from \cite{Zumbrun}, that is $|r^1|=1$ eigenvector as before (referring to maximal/minimal decay/growt rate of $B^\pm$) and for $k>0$
\begin{equation*}
r^{k+1}=P^{k+1}r^k.
\end{equation*}
Using a careful analysis of the Evans function at infinity one can show numerically that  there are no eigenvalues for  $|\lambda|\geq 1.9 \cdot 10^4$. Here we complement this information with a numerical evidence of absence of eigenvalues inside that circle.\\
We have to initialize the computation on the real axis and for stability reasons for a not very large values of $\lambda$. For this we use two contours. One semi-circular contour with radius 10, center at $\lambda = 0$ and vertical segment on the imaginary axis. We don't evaluate the Evans function at $0$, but evaluate it till $i 10^{-6}$. Along this contour we integrate the Kato ODE using $10^4$ points. We also use another countour which surrounds a semi-annular region in the right half-plane with two semi-circles with radii $5$ and $1.9 \cdot 10^4$, center at $\lambda = 0$ and vertical segment on the imaginary axis. Along this contour we use $10^7$ points to integrate the Kato ODE. We use higher density near the origin.\\
Then using these initial conditions we compute $E(\lambda)$ with the stiff solver ode15s in matlab, with relative tolerance $10^{-4}$ and we set $L_1=40$. Finally, we apply the symmetry of $E(\bar{\lambda})=\overline{E(\lambda)}$. The union of the areas, surrounded by the two contours covers the whole region of the unstable half-plane inside $|\lambda| \leq 1.9 \cdot 10^4$. The Evans function $E(\lambda)$ is plotted in Figure \ref{figure_evans_fun_2} and its winding number is (approximately) 0, giving a numerical evidence of point spectrum stability.\\
Finally, we note that  if $\lambda \in \mathbb{R}$, then $E(\lambda) \in \mathbb{R}$. Our numerics agrees with this simple observation, because we get  $\Im{E(\lambda)} \approx 0$ for $\lambda = 1.9\cdot10^4$. Moreover, to corroborate this accuracy, we also use the Cauchy integral formula
\begin{equation*}
E(a) = \frac{1}{2 \pi i}\int_{\Gamma}\frac{E(z)}{z-a}dz
\end{equation*}
for $a$ inside the contour $\Gamma$ surrounding the semi-annular region, and,  for $a = 1.9\cdot10^4-20$, we get a relative error less than $5\cdot10^{-4}$. \\
Moreover, we present a computation of the Evans function along a contour, surrounding a semi-annular region with radii $10^{-6}$ and $1.9 \cdot 10^4$, center at $\lambda = 0$ and vertical segment on the imaginary axis. Along the contour we integrate the Kato ODE with $10^7$ points. The winding number of the Evans function is zero. See Figure \ref{evans_function_annular_small_circle}. Using the Cauchy integral formula with $a = 1.9\cdot10^4-20$, we observe a relative error less than $5\cdot10^{-4}$. 
\captionsetup[subfigure]{labelformat=empty}
\begin{figure}
\begin{center}
\begin{subfigure}[b]{1.0\linewidth}
\centering
\includegraphics[scale=0.6]{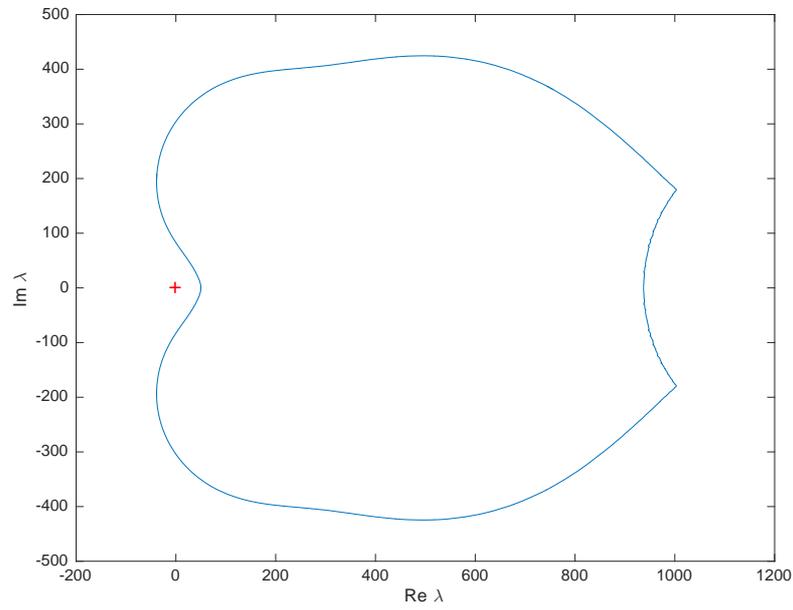}
\caption{(a)}
\end{subfigure}
\begin{subfigure}[b]{1.0\linewidth}
\centering
\includegraphics[scale=0.6]{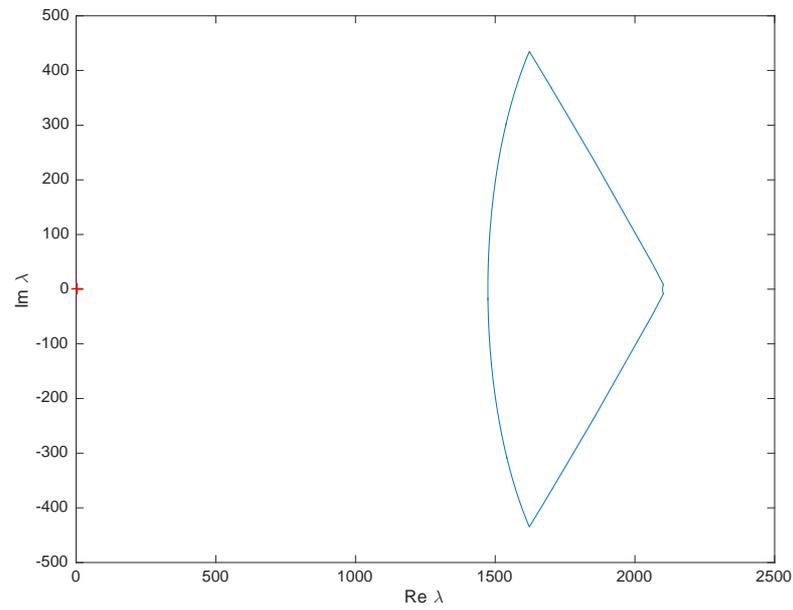}
\caption{(b)}
\end{subfigure}
\caption{(a) The image of a semi-circular contour with radius $10$ through the Evans function $E(\lambda)$. (b) The image of a contour, surrounding a semi-annular region with radii $5$ and $1.9\cdot10^4$ through $E(\lambda)$. The origin is marked in red.}
\label{figure_evans_fun_2}
\end{center}
\end{figure}

\begin{figure}
\centering
\includegraphics[scale=0.6]{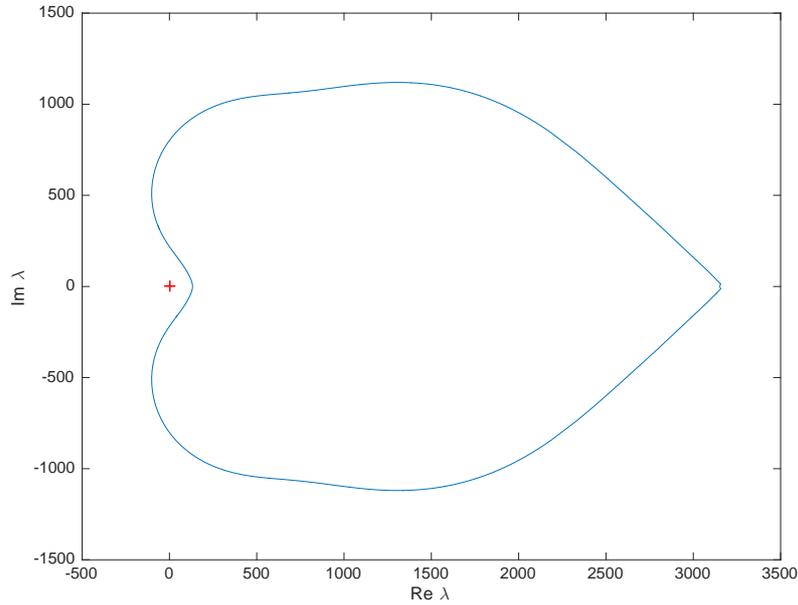}
\caption{The image of a contour, surrounding a semi-annular region with radii $10^{-6}$ and $1.9\cdot10^4$ through $E(\lambda)$. The origin is marked in red.}
\label{evans_function_annular_small_circle}
\end{figure}

\FloatBarrier


\begin{thebibliography}{10}
\bibitem{AM1}  P. Antonelli, P. Marcati, 
On the finite energy weak solutions to a system in Quantum Fluid Dynamics , \textit{Comm. Math. Phys.} 287, 657-686 (2009)
\bibitem{AM2} P. Antonelli, P. Marcati,  The Quantum Hydrodynamics system in two space dimensions, \textit{Arch. Ration. Mech. Anal.} 203 , 499-527  (2012)
\bibitem{AMtf} P. Antonelli, P. Marcati, Finite Energy Global Solutions to a Two-Fluid Model Arising in Superfluidity, \textit{Bull. Inst. Math. Acad. Sin.} 10,349--373 (2015)
\bibitem{AMDCDS} P. Antonelli, P. Marcati, Quantum hydrodynamics with nonlinear interactions, \textit{Discrete Contin. Dyn. Syst.} Ser. S 9, 1--13 (2016)
\bibitem{AS} P. Antonelli, S. Spirito, Global existence of finite energy weak solutions of quantum Navier-Stokes equations, \textit{Arch. Ration. Mech. Anal.} 225, 1161-1199 (2017)

\bibitem{Michele1}
F. Di Michele, P. Marcati, B. Rubino, Steady states and interface transmission conditions for heterogeneous quantum-classical 1-D hydrodynamic model of semiconductor devices. Physica D: Nonlinear Phenomena, 243(1), pp. 1-13, 2013 

\bibitem{Michele}
F. Di Michele, P. Marcati, B. Rubino, Stationary solution for transient quantum hydrodynamics with bohmenian-type boundary conditions, Computational and Applied Mathematics, 36(1), pp. 459-479, 2017

\bibitem{DFM}D. Donatelli, E. Feireisl, P. Marcati, Well/ill posedness for the Euler- Korteweg-Poisson system and related problems, \textit{Comm. Partial Differential Equations}, 40, 1314-1335 (2015)
\bibitem{DM1}D. Donatelli, P. Marcati, Quasineutral limit, dispersion and oscillations for Korteweg type fluids, \textit{SIAM J. Math. Anal.} 47, 2265-2282  (2015)
\bibitem{DM} D. Donatelli, P. Marcati,  Low Mach number limit for the quantum hydrodynamics system, \textit{Res. Math. Sci.} 3, 3-13  (2016)
\bibitem{GLT} J. Giesselmann, C. Lattanzio, and  A.E. Tzavaras,  Relative Energy for the Korteweg Theory and Related Hamiltonian Flows in Gas Dynamics, \textit{Arch. Ration. Mech. Anal.} 223 , 1427-1484  (2017)
\bibitem{Gurevich1}A. V. Gurevich and A. P. Meshcherkin. Expanding self-similar discontinuities and shock waves in dispersive hydrodynamics, \textit{Sov. Phys. JETP}, 60(4):732-740, 1984.
\bibitem{Gurevich}
A. V. Gurevich and L. P. Pitaevskii, Nonstationary structure of a collisionless shock wave, \textit{Sov. Phys. JETP}, 38:291-297 (1974)
\bibitem{Hoefer}M. A. Hoefer, M. J. Ablowitz, I. Coddington, E. A. Cornell, P. Engels, and V. Schweikhard, Dispersive and classical shock waves in Bose-Einstein condensates and gas dynamics, \textit{Phys. Rev. A}, 74:023623 (2006)

\bibitem{Hoefer1}M. A. Hoefer, Shock Waves in Dispersive Eulerian Fluids, \textit{J. Nonlinear Sci.}, Volume 24, Issue 3, pp 525-577, (2014)

\bibitem{Humpherys}J. Humpherys, On the shock wave spectrum for isentropic gas dynamics with capillarity, \textit{J. Differential Equations}, 246(7):2938-2957 (2009)

\bibitem{Kapitula}T. Kapitula, K. Promislow, \textit{Spectral and Dynamical Stability of Nonlinear Waves}, Springer-Verlag New York, 2013

\bibitem{Kato}T. Kato, \textit{Perturbation Theory for Linear Operators}, second edition, Grundlehren der Mathematischen Wissenschaften, Springer-Verlag, Berlin-New York, 1976.

\bibitem{Zhelyazov}
C. Lattanzio, P. Marcati, D. Zhelyazov, Dispersive shocks in Quantum Hydrodynamics with viscosity, arXiv preprint arXiv:1812.10279 (2019).

\bibitem{Nov}S. Novikov, S. V. Manakov, L. P. Pitaevskii, and V. E. Zakharov, \textit{Theory of Solitons}, Consultants Bureau, New York, 1984.

\bibitem{Sagdeev}S.R. Z. Sagdeev,  Kollektivnye protsessy i udarnye volny v razrezhennol plazme (Collective processes and shock waves in a tenuous plasma), in: \textit{Voprosy teorii plazmy (Problems of Plasma Theory)}, Vol. 5, Atomizdat, 1964.

\bibitem{Sandstede}B. Sandstede, Stability of Travelling Waves, \textit{Handbook of Dynamical Systems II}, edited by B. Fiedler, Elsevier (2002), 983-1055

\bibitem{Zak}V. E. Zakharov, Stability Of Periodic Waves Of Finite Amplitude On The Surface Of A Deep Fluid, \textit{Zhurnal Prildadnoi Mekhaniki i Tekhnicheskoi Fiziki}, 9(2), 86-94 (1968)

\bibitem{Zumbrun}K. Zumbrun, A local greedy algorithm and higher order extensions for global numerical continuation of analytically varying subspaces, \textit{Quart. Appl. Math.} Vol. 68, No. 3, pp. 557-561 (2010)
\end{thebibliography}
\end{document}